\documentclass[12pt]{amsart}

\usepackage{amssymb}

\usepackage{xcolor}        %lots of colors

\usepackage{hyperref} 
\hypersetup{colorlinks=true, linkcolor={red!50!black}, citecolor={green!50!black}}

\newtheorem{theorem}{Theorem}[section]
\newtheorem{cor}[theorem]{Corollary}
\newtheorem{lemma}[theorem]{Lemma}
\theoremstyle{definition}
\newtheorem{example}[theorem]{Example}
\newtheorem{remark}[theorem]{Remark}

\newcommand{\CC}{\mathfrak{C}}

\newcommand{\AddrSet}{\Omega}

\newcommand{\cone}[1]{#1\CC}
\newcommand{\nbd}{\nobreakdash-}
\newcommand{\prefix}{\preccurlyeq}
\newcommand{\set}[2]{\{\,#1 : #2\,\}}
\newcommand{\spt}[1]{\operatorname{Supt}(#1)}

\newcommand{\sw}[2]{ (\hspace{1pt}#1\;#2\hspace{1pt} )}
\newcommand{\seteq}{:=}

\title[Maximality of~$T$ in~$V$]{The maximality of~$T$ in Thompson's group~$V$}
\author{J. Belk, C. Bleak, M. Quick, and R. Skipper}


\begin{document}

\begin{abstract}
We show that R.\ Thompson's group~$T$ is a maximal subgroup of the group~$V$.  The argument provides examples of foundational calculations which arise when expressing elements of $V$ as products of transpositions of basic clopen sets in Cantor space~$\CC$.
\end{abstract}

\maketitle

\section{Introduction}

In this paper, we prove the foundational result that Thompson's group $T$ is a maximal subgroup of Thompson's group~$V$.  We were surprised to learn that this result has not previously appeared in the literature.  In our view, Higman's famous quote \cite{Higman1990} 
 applies: 

\begin{quote}
This paper can, I think, best be considered as a sort of five finger exercise. In it we polish the techniques from a corner of group theory\ldots
\end{quote}
where our corner of group theory explores the connection between the dynamics of an action and corresponding algebraic structures.

We shall work with Richard Thompson's group~$V$ viewed as a group of homeomorphisms of Cantor space $\CC = \{0,1\}^{\omega}$.  We view points of Cantor space as infinite sequences $w = x_{1}x_{2}\dots$ where each~$x_{i}$ is a letter from the alphabet~$\{0,1\}$.  We recommend the paper~\cite{CFP} as the standard introductory reference to the Thompson groups, where, in particular, the subgroup~$T$ of~$V$ is defined as consisting of those transformations that can be viewed as maps on the circle.  The reader may also find it convenient to consult the paper~\cite{BleakQuick} for the permutation notation that we employ here.  In particular, in the latter source, it is described how the transpositions of basic clopen subsets of~$\CC$ (that we term ``swaps'') generate the group~$V$.

An important theme throughout the study of groups is to describe their maximal subgroups.  In~\cite{BBQS2024}, we described technology to construct many examples of maximal subgroups of Thompson's group~$V$, and we asserted that $T$ is a maximal subgroup of $V$ that does not arise from this construction.  We subsequently learned that the maximality of $T$ in $V$ was not a known result.

Note that Thompson's group $F$ is maximal in $T$.  Indeed, since $F$ is the stabiliser of $0$ in the circle and $T$ acts 2-transitively on the orbit of $0$ in the circle, $F$ has only two double cosets in~$T$.  No similar statement holds for $T$ in $V$.  Indeed, it is not hard to show that the identity, the swap $\sw{00}{01}$, and the swap $\sw{00}{10}$
lie in different double cosets of $T$ in $V$.  (See Section~\ref{sec:prelim} for the definition of swaps.)

Section~\ref{sec:prelim} contains the preliminaries needed and describes how we handle the prefix substitutions in~$V$.  The proof of the theorem is conducted in Section~\ref{sec:arg}.

\subsection*{Acknowledgements}
We would like to thank Matthew G. Brin who asked us if we knew of a reference for our claimed result (in \cite{BBQS2024}) that $T$ was maximal in $V$, and encouraged us to write a proof of this claim. We also thank Pierre-Emmanuel Caprace for interesting comments on an earlier draft.

The fourth author was partially funded by NSF DMS-2005297.

For the purpose of open access, the authors have applied a CC BY public copyright Licence to any Author Accepted Manuscript version arising.

\section{Preliminaries}
\label{sec:prelim}

We shall write $\AddrSet = \{0,1\}^{\ast}$ for the set of finite words in the alphabet~$\{0,1\}$ and refer to the elements of~$\AddrSet$ as \emph{addresses}.  If $\alpha \in \AddrSet$, then this determines a \emph{cone}~$\cone{\alpha}$, which is the basic clopen subset of Cantor space~$\CC$ consisting of those elements with prefix~$\alpha$:
\[
\cone{\alpha} = \set{\alpha w}{w \in \CC}.
\]
We shall refer to~$\cone{\alpha}$ as the \emph{cone with address~$\alpha$} or, more simply, the \emph{cone at~$\alpha$}.  If $\alpha$~and~$\beta$ are addresses in~$\AddrSet$, we write $\alpha \prefix \beta$ when $\alpha$~is a prefix of~$\beta$.  We also write $\alpha \perp \beta$ when both $\alpha \not\prefix \beta$ and $\beta \not\prefix \alpha$ and, in this case, we say that $\alpha$~and~$\beta$ are \emph{incomparable}.  We shall also make use of the \emph{lexicographic order} on addresses.  We write $\alpha < \beta$ when $\alpha = a_{1}a_{2}\dots a_{r}$ is earlier than $\beta = b_{1}b_{2}\dots b_{s}$ in this order; that is, when, for some~$k$, \ $a_{1} = b_{1}$, \dots, $a_{k-1} = b_{k-1}$, $a_{k} = 0$ and $b_{k} = 1$.  There is similarly a lexicographic order~$<$ defined upon Cantor space.

Let $\cone{\alpha_{1}}$,~$\cone{\alpha_{2}}$, \dots,~$\cone{\alpha_{n}}$ and $\cone{\beta_{1}}$,~$\cone{\beta_{2}}$, \dots,~$\cone{\beta_{n}}$ be two collections of disjoint cones that each form a partition of~$\CC$.  We may define a map $g \colon \CC \to \CC$ by $\alpha_{i}w \mapsto \beta_{i}w$ for $i = 1$,~$2$, \dots,~$n$ and all $w \in \CC$.  We call such a map a \emph{prefix substitution map} and Thompson's group~$V$ may be defined as the group of all prefix substitution maps with composition as binary operation.  We may describe this prefix substitution map~$g$ by a \emph{paired tree diagram}, or \emph{tree-pair},~$(D,R,\sigma)$ where $D$~and~$R$ are finite binary rooted trees with $n$~leaves and $\sigma$ is a permutation of the points in~$\{1,2,\dots,n\}$.  The leaves of the domain tree~$D$ correspond to the addresses~$\alpha_{i}$ and are labelled $1$,~$2$, \dots~$n$ from left-to-right, while the leaves of the range tree~$R$ correspond to the addresses~$\beta_{j}$ and are labelled $1\sigma$,~$2\sigma$, \dots,~$n\sigma$ from left-to-right in such a way that if $\alpha_{i}$~corresponds to the leaf with some label~$k$ then $\beta_{i}$~is also labelled~$k$.  Note that if $\alpha_{1}$,~$\alpha_{2}$, \dots,~$\alpha_{n}$ are the addresses of the leaves of the rooted tree~$D$ appearing from left-to-right in its diagram then, in the lexicographic order, $\alpha_{1} < \alpha_{2} < \dots < \alpha_{n}$.

An element $g \in V$ has finite order if and only if we can represent it by a tree-pair~$(D,D,\sigma)$ where the domain and range trees are the same.  The permutation~$\sigma$ then represents a permutation of the leaves of~$D$.  (See \cite{CentralizersInV2013,BCST} for different proofs of this.)  We shall call such elements \emph{permutations} in~$V$.

In our argument, we make particular use certain prefix substitutions that we call \emph{swaps}.  If $\alpha$~and~$\beta$ are incomparable addresses (so the cones $\cone{\alpha}$~and~$\cone{\beta}$ are disjoint), we define the prefix substitution~$\sw{\alpha}{\beta}$ by
\[
p \cdot \sw{\alpha}{\beta} = \begin{cases}
p &\text{if $p \notin \cone{\alpha} \cup \cone{\beta}$,} \\
\beta w &\text{if $p = \alpha w$ for some $w \in \CC$,} \\
\alpha w &\text{if $p = \beta w$ for some $w \in \CC$.}
\end{cases}
\]
This swap is an element of~$V$ that has the effect of interchanging the cones $\cone{\alpha}$ and~$\cone{\beta}$.  We think of them as analogous to the familiar transpositions in symmetric groups.  It was observed in \cite{BleakQuick} that the swaps form a convenient generating set for $V$ and this leads to a ``Coxeter-type'' infinite presentation for $V$.  As in \cite{BleakQuick}, we extend this ``permutation of disjoint cones'' notation in the natural fashion for finite order elements of $V$ as products of disjoint cycles of cones, allowing us to write an element such as $(00\; 110\; 010\; 101 ) \, (011\; 111 )$ hopefully without inducing confusion in the reader.  Note that $\sw{\alpha}{\beta} = \sw{\alpha0}{\beta0} \, \sw{\alpha1}{\beta1}$ whenever $\alpha$~and~$\beta$ are addresses with $\alpha \perp \beta$.  Through this observation, it follows that every element of~$V$ can be written as an even number of swaps, hinting towards the simplicity of~$V$.

Finally, if $g \in V$, we define the \emph{support} of~$g$ to be the set
\[
\spt{g}\seteq \set{p\in \CC}{pg\neq p}
\]
and note that this is always a clopen subset of~$\CC$.  If $\spt{g} \neq \CC$, then we say that the map~$g$ is of \emph{small support}.

\section{The argument}
\label{sec:arg}

We noted in our preliminaries that an element~$g \in V$ is of finite order is represented by a tree-pair~$(D,D,\sigma)$.  For such a permutation~$g$, we shall say that \emph{$g$~admits an interleaved representation} if the tree~$D$ can be chosen so that if $\alpha$~is any leaf of~$D$ representing a cone~$\cone{\alpha}$ that is moved by~$g$, then the left-neighbour and right-neighbour leaves of~$\alpha$ (in the natural circular order on the leaves of~$D$) both represent cones that are fixed (necessarily pointwise) by~$g$.  In this case, we also call~$g$ an \emph{interleaved permutation}.  We also use the term \emph{interleaved swap} for a swap $\sw{\alpha}{\beta}$ that is an interleaved permutation.  Within our argument, we shall use interleaved permutations within various products to build other permutations.

\begin{example}
\label{ex:interleavedSwapProduct}
Let $\alpha,\beta\in \AddrSet$ with $\alpha\perp\beta$.  The formula
\[
\sw{\alpha}{\beta} = \sw{\alpha0}{\beta0}\,\sw{\alpha1}{\beta1}
\]
expresses the swap~$\sw{\alpha}{\beta}$ as a product of two swaps of small support, both of which are interleaved permutations.
\end{example}

Given a finite rooted binary tree~$A$, we shall say that a sequence $(\alpha_i)_{i=1}^n$ of leaves of~$A$ is \emph{interleaved} if every distinct pair $\alpha_{i}$~and~$\alpha_{j}$ of leaves are not neighbours in the natural left-to-right order on the leaves of~$A$.  We will say the sequence is \emph{ordered} if, for every such pair of indices with $i<j$, the leaf~$\alpha_i$ is to the left of the leaf~$\alpha_j$ (that is, $\alpha_{i} < \alpha_{j}$ in the lexicographic order on addresses).

We leave the following lemma to the reader.  The proof only requires techniques which appear in \cite{CFP}.

\begin{lemma}[The Interleaved Shaping Lemma]\label{lem:shapingInT}
Let $(\alpha_i)_{i=1}^k$ and $(\beta_i)_{i=1}^k$ be ordered, interleaved sequences of $k$ leaves on some finite rooted binary trees $A$ and $B$ respectively. Then there are finite rooted binary trees $D$ and $R$ with the same number of leaves and a bijection $\tau$ from the leaves of $D$ to the leaves of $R$ so that the tree-pair $(D,R,\tau)$ represents an element $g\in T$ such that, for each index~$i$, $\alpha_i$~is a leaf of~$D$, $\beta_i$~is a leaf of~$R$, and $\alpha_i\tau = \beta_i$.
\end{lemma}

\begin{cor}\label{cor:interleavedSwaps}
If $\alpha,\beta\in\AddrSet$ with $\alpha\perp\beta$ such the swap $\sw{\alpha}{\beta}$ is an interleaved permutation, then there exists $g\in T$ such that $\sw{\alpha}{\beta}^g = \sw{00}{10}$.  In particular, any two interleaved swaps in~$V$ are conjugate to each other using an element of~$T$.
\end{cor}

\begin{lemma}\label{lem:SwampWithT}
If $\sw{\alpha}{\beta}$ is a swap of small support then $ {\langle} \,\sw{\alpha}{\beta},T\, {\rangle} = V$.
\end{lemma}

\begin{proof}
Let $(D,D,\sigma)$ be a tree-pair representing $\sw{\alpha}{\beta}$ with $\alpha$~and~$\beta$ being leaves in the tree~$D$. As $\sw{\alpha}{\beta}$~is a swap of small support the tree $D$ has at least three leaves.  If necessary, replace $D$ by a new tree $D$ with at least four leaves by splitting a leaf of the original tree $D$ that represents a cone not in the support of~$\sw{\alpha}{\beta}$. 

If it is the case that $\alpha$~and~$\beta$ are adjacent leaves of~$D$, then conjugate~$\sw{\alpha}{\beta}$ by an element $g \in T$ with representative tree-pair $(D,D,\rho)$ where $\rho$ is a cyclic rotation of the leaves of~$D$ so that $\sw{\alpha}{\beta}^g = \sw{\beta}{\gamma}$ where the leaf~$\gamma$ represents the cone~$\cone{\gamma}$ that is disjoint from the support of $\sw{\alpha}{\beta}$ and, by our assumption that $D$~has at least four leaves, $\cone{\alpha} \cup \cone{\beta} \cup \cone{\gamma} \neq \CC$.  We now compute that
\[
\sw{\alpha}{\beta}^{\sw{\beta}{\gamma}}=\sw{\alpha}{\gamma}.
\]
Either $\sw{\beta}{\gamma}$ or $\sw{\alpha}{\gamma}$ is an interleaved swap.  We may therefore assume that $\sw{\alpha}{\beta}$ was an interleaved swap in the first place.

The subgroup generated by~$\sw{\alpha}{\beta}$ and~$T$ then contains all interleaved swaps by Corollary~\ref{cor:interleavedSwaps} and hence all swaps by Example~\ref{ex:interleavedSwapProduct}.  Therefore this subgroup equals~$V$.
\end{proof}

\begin{remark}
Note that in the above lemma we cannot remove the stated restriction to a swap of small support since $\sw{0}{1}$ is actually an element of $T$ and $T$ is a proper subgroup of~$V$.
\end{remark}

\begin{lemma}\label{lem:VFromThreeAndT}
If $ (\alpha_1\;\alpha_2\;\alpha_3 )$ is a $3$\nbd cycle of small support then
\[
\langle (\alpha_1\;\alpha_2\;\alpha_3 ),T \rangle = V.
\]
\end{lemma}

\begin{proof}
First we may assume (via conjugation arguments similar to those used in Lemma~\ref{lem:SwampWithT}) that  $ (\alpha_1\;\alpha_2\;\alpha_3 )$ is an interleaved permutation.  Then employing Lemma \ref{lem:shapingInT} we may find an element $t \in T$ whose definition includes the prefix substitutions $\alpha_{1} \mapsto 00$, $\alpha_{2} \mapsto 100$ and $\alpha_{3} \mapsto 1100$.  Then
\[
(\alpha_1\;\alpha_2\;\alpha_3)^{t} = (00\;100\;1100).
\]
It therefore remains to show that the subgroup $H = \langle (00\;100\;1100), T \rangle$ is equal to~$V$.

Firstly, there is an element $x\in T$ such that 
\[
(00\;100\;1100)^x =  (00\;100\;1110)
\]
so $H$~contains
\[
\sw{00}{1110} \sw{100}{1100}= (00\;100\;1100) (00\;100\;1110).
\]
Now conjugate this product of disjoint swaps by $(00\;100\;1100)$ to see that $H$~also contains
\[
\bigl( \sw{00}{1110} \sw{100}{1100} \bigr)^{(00\;100\;1100)} = \sw{00}{1100} \sw{100}{1110}.
\]
Observe that this is an interleaved permutation and that the addresses corresponding to its support are $00 < 100 < 1100 < 1110$ in the lexicographic order.  Hence, by use of Lemma~\ref{lem:shapingInT}, there are elements $y,z\in T$ such that
\[
\bigl( \sw{00}{1100}\sw{100}{1110} \bigr)^y = \sw{000}{1000}\sw{010}{1010}\]
and
\[
\bigl(\sw{00}{1100}\sw{100}{1110}\bigr)^z = \sw{001}{1001}\sw{011}{1011}.
\]
The subgroup~$H$ therefore contains the product of these last two elements, namely the swap~$\sw{0}{10}$.  Hence $\sw{0}{10}\in  \langle (\alpha_1\;\alpha_2\;\alpha_3 ),T \rangle$, and so by Lemma \ref{lem:SwampWithT} we deduce $\langle (\alpha_1\;\alpha_2\;\alpha_3 ),T \rangle = V$.
\end{proof}

\begin{theorem}
\label{thm:max}
Thompson's group $T$ is a maximal subgroup of Thompson's group~$V$.
\end{theorem}

\begin{proof}
Let $a \in V\backslash T$.  Representing $a$ by a tree-pair with some domain tree and some range tree, there is an element $b$ of~$T$ that takes the range tree of~$a$ to the domain tree while returning at least one image cone to its original location.  The product $c \seteq ab$ is then a permutation in~$V$ which fixes at least one cone and which is not in~$T$.   Through conjugation by an element of $T$ we may assume that $c$~fixes the cone at~$0$ and that there is a tree-pair $(D,D,\sigma)$ representing $c$ so that $0$ and $10$ are leaves of $D$ and that the cone $\cone{10}$ is moved to some other cone to the right of~$\cone{10}$ by $c$ (that is, to a cone $\cone{11\beta}$ for some~$\beta$, where $11\beta$~is also a leaf of~$D$).

Let us write $c$ as a product of its disjoint cycles of cones as follows:
\[
c = \prod_{i=1}^{r}  (\alpha_{i1}\;\alpha_{i2}\;\cdots \; \alpha_{ik_i})
\]
where we may assume $\alpha_{11}=10$ and that all addresses $\alpha_{ij}$ represent leaves of the tree $D$ and so in particular have~$1$ as prefix.
(We refer to this as a cycle decomposition of the permutation~$c$.  As is usual, we omit any cycles of length one.)

Let $\alpha_{ij}$ be the address appearing in the cycle decomposition of~$c$ that is rightmost as a leaf in the tree~$D$ among such addresses.  (Thus the cone $\cone{\alpha_{ij}}$ is moved by~$c$ to a cone $\cone{\alpha_{ik}}$ that is to the left of $\cone{\alpha_{ij}}$; that is, $\alpha_{ik} < \alpha_{ij}$ in the lexicographic order on addresses.)  Further, let $\gamma$ be the next leaf of $D$ to the right of $\alpha_{ij}$ in the cyclic order on the leaves of $D$, so necessarily $c$ fixes the cone $\cone{\gamma}$ pointwise.  (If it were the case that $\alpha_{ij}$~is the address of the rightmost leaf of the tree~$D$, then $\gamma = 0$, the leftmost leaf of~$D$.)  Now let $d$ be an element of $T$ with tree pair $(D,R,\tau)$ where the tree~$R$ has both $10$ and $11$ as leaves and where $d$ takes $\cone{\alpha_{ij}}$ to the cone~$\cone{10}$ and $\cone{\gamma}$ to the cone~$\cone{11}$.

 The conjugate element $e\seteq c^{d}$ is then an element of~$V$ which is a permutation on the leaves of the tree~$R$, which fixes the points in the cone~$\cone{11}$ and has a cycle decomposition with the address~$10$ as an entry in one of the cycles.  Consequently, all other addresses in this cycle decomposition have~$0$ as prefix.

From this construction a straightforward computation shows the commutator $[c,e]$ is the $3$\nbd cycle $f\seteq (\beta_1\;\beta_2\;\beta_3)$ with $\beta_{1} = 10$, $\beta_{2}$~equal to the image of~$10$ under~$e$ and $\beta_{3}$~equal to the image of~$10$ under~$c$.  In particular, $\beta_{2}$~has $0$~as prefix and $\beta_{3}$~has $1$~as prefix.  Consequently, if $f$~were to have full support, then necessarily $\beta_{2} = 0$ and $\beta_{3} = 11$.  Since $c$~fixes the points in the cone~$\cone{0}$, this forces $c = \sw{10}{11}$.  Then $V = \langle c, T \rangle$ by Lemma~\ref{lem:SwampWithT} and $V = \langle a,T \rangle$ follows.

Let us suppose then that $f$~does not have full support.  In this case, $V = \langle f,T \rangle$ by Lemma~\ref{lem:VFromThreeAndT}.  Since $f \in \langle a,T \rangle$, we conclude $V = \langle a,T \rangle$.  This shows that $T$~is indeed a maximal subgroup of~$V$.
\end{proof}

\bibliographystyle{amsplain}
\bibliography{bibstuff}

\newpage
\noindent
James Belk, \texttt{jim.belk@glasgow.ac.uk}\\
School of Mathematics \& Statistics, University of Glasgow, Glasgow, UK\\[5pt]
Collin Bleak, \texttt{collin.bleak@st-andrews.ac.uk}\\
School of Mathematics \& Statistics, University of St Andrews, St Andrews, UK\\[5pt]
Martyn Quick, \texttt{mq3@st-andrews.ac.uk}\\
School of Mathematics \& Statistics, University of St Andrews, St Andrews, UK\\[5pt]
Rachel Skipper, \texttt{rachel.skipper@utah.edu}\\
Department of Mathematics, University of Utah, Salt Lake City, Utah, USA

\end{document}